\documentclass[12pt,reqno]{amsart}
\usepackage[a4paper,margin=1in]{geometry}
\usepackage[T1]{fontenc}
\usepackage[utf8]{inputenc}
\usepackage{lmodern}
\usepackage{microtype}
\usepackage{amsmath,amssymb,amsthm,mathtools}
\usepackage{enumitem}
\usepackage{xcolor}
\usepackage[colorlinks=true,linkcolor=blue!55!black,citecolor=green!38!black,urlcolor=blue!55!black]{hyperref}

\numberwithin{equation}{section}

\newtheorem{theorem}{Theorem}[section]
\newtheorem{proposition}[theorem]{Proposition}
\newtheorem{lemma}[theorem]{Lemma}

\theoremstyle{definition}

\theoremstyle{remark}
\newtheorem{remark}[theorem]{Remark}

\DeclareMathOperator{\argmin}{argmin}

\DeclareMathOperator{\dist}{dist}
\DeclareMathOperator{\spec}{Spec}
\DeclareMathOperator{\tr}{tr}
\let\Re\relax
\DeclareMathOperator{\Re}{Re}
\let\Im\relax
\DeclareMathOperator{\Im}{Im}

\newcommand{\R}{\mathbb R}
\newcommand{\C}{\mathbb C}
\newcommand{\T}{\mathbb T}

\newcommand{\calD}{\mathcal D}
\newcommand{\calG}{\mathcal G}
\newcommand{\calM}{\mathcal M}
\newcommand{\calN}{\mathcal N}
\newcommand{\calO}{\mathcal O}
\newcommand{\calU}{\mathcal U}
\newcommand{\calV}{\mathcal V}
\newcommand{\calX}{\mathcal X}

\newcommand{\OmegaCell}{\mathbb{R}^d/\Gamma}

\title[Generic isolation and nondegeneracy]
{Band edges of periodic Schr\"odinger operators are generically isolated and nondegenerate}

\author{Zhongkai Tao}
\address[Zhongkai Tao]{Institut des Hautes \'Etudes Scientifiques, 35 route de Chartres, 91440 Bures-sur-Yvette, France}
\email{ztao@ihes.fr}

\author{Mengxuan Yang}
\address[Mengxuan Yang]{Department of Mathematics, Texas A\&M University, College Station, TX 77843, USA}
\email{yangmx@tamu.edu}
\date{}

\begin{document}

\begin{abstract}
For periodic Schr\"odinger operators $H_V=-\Delta+V$ with bounded real-valued potentials on $\mathbb R^d$ with $d\ge2$, we show that for generic potentials, each endpoint of every spectral gap is attained by a single Bloch band at only finitely many quasimomenta, and has a nondegenerate Hessian at every attaining point. This proves the Spectral Edge Conjecture for periodic Schr\"odinger operators.
\end{abstract}
\maketitle

\section{Introduction}
\label{sec:intro}

Let $\Gamma\subset\R^d$, $d\geq2$, be a lattice of full rank, with the dual lattice and the Brillouin zone given by
\begin{equation}
\label{eq:lattices}
  \Gamma^*
  =\{\gamma^*\in\R^d:\gamma^*\cdot\gamma\in2\pi\mathbb Z
      \text{ for all }\gamma\in\Gamma\},
 \quad
  \T_*^d=\R^d/\Gamma^*.
\end{equation}
We consider Schr\"odinger operators $H_V=-\Delta+V: L^2(\R^d) \to L^2(\R^d)$ with bounded real-valued $\Gamma$-periodic potentials in the space $\calX=L^\infty(\R^d/\Gamma;\R)$.
For $V\in\calX$, Floquet--Bloch theory gives the Floquet--Bloch transformed operators
\begin{equation}
\label{eq:fiber-operator}
  H_V(k)=(-i\nabla+k)^2+V: \  L^2(\OmegaCell) \to L^2(\OmegaCell)
 \quad k\in\T_*^d,
\end{equation}
with domain $H^2(\OmegaCell)$ and eigenvalues (counted with multiplicity)
\begin{equation}
\label{eq:ordered-eigenvalues}
  E_1(V,k)\leq E_2(V,k)\leq\cdots,\quad k\in\T_*^d.
\end{equation}
Then by Floquet--Bloch theory, we have the decomposition of the continuous spectrum
\begin{equation}
\label{eq:floquet-spectrum}
  \spec H_V=\bigcup_{n\geq1}E_n(V,\T_*^d).
\end{equation}
A subset of $\calX$ is called \emph{residual} if it contains a countable intersection of open dense subsets of $\calX$. The endpoints $a,b\in\spec H_V$ of a bounded connected component $(a,b)$ of $\mathbb R\setminus\spec H_V$ are called \emph{band edges} (or \emph{spectral edges}), and the interval $(a,b)$ is called a \emph{spectral gap}.\footnote{The bottom of the spectrum is well understood in the literature (cf.~\cite{KirschSimon}), thus we exclude it from the definition of the band edges.} 
Our main result is the following
\begin{theorem}
\label{thm:main}
There exists a residual set $\calG\subset\calX$ such that the following holds. For every $V\in\calG$ and spectral gap $(a,b) \subset \R\setminus\spec H_V$, 
\begin{enumerate}[label=\textup{(\roman*)},leftmargin=3em]
  \item there exist unique indices $n_-,n_+\geq1$ such that
  \[
    a\in E_{n_-}(V,\T_*^d),
   \quad
    b\in E_{n_+}(V,\T_*^d);
  \]
  for every $k$ with $E_{n_-}(V,k)=a$, the number $a$ is a simple eigenvalue of $H_V(k)$, and for every $k$ with $E_{n_+}(V,k)=b$, the number $b$ is a simple eigenvalue of $H_V(k)$;
  \item the sets
  \begin{equation}
      \label{eq:extrema}
          \{k \in \mathbb{T}_*^d:E_{n_-}(V,k)=a\},\quad \{k \in \mathbb{T}_*^d:E_{n_+}(V,k)=b\}
  \end{equation}
  are finite;
  \item for every $k_-$ in the first set and $k_+$ in the second set in \eqref{eq:extrema}, we have
  \[
    d_k^2E_{n_-}(V,k_-)<0,\quad d_k^2E_{n_+}(V,k_+)>0.
  \]
\end{enumerate}
\end{theorem}
The theorem proves the Spectral Edge Conjecture (cf.~\cite[Conjecture~8.5]{Kuchment2023}) for periodic Schr\"odinger operators, which predicts that, for generic potentials, any band edge belongs to a single band function $E_n(k)$, is attained at isolated quasimomenta $k$, and has a nondegenerate Hessian there. The structure of band edges of periodic Schr\"odinger operators is a fundamental open question in mathematical physics. Filonov--Kachkovskiy \cite{FilonovKachkovskiy} proved that global band extrema are isolated for a broad class of periodic elliptic operators when $d=2$. However, their argument does not establish nondegeneracy of the Hessian and is specific to dimension two. 

The Spectral Edge Conjecture has several important consequences. The inverse of the Hessian is used to define the effective mass in solid state physics~\cite[Chapter 12, equation~(12.29)]{AshcroftMermin1976}. The nondegeneracy of the band edge is also a fundamental hypothesis in results on Liouville theorems and Green function asymptotics~\cite{KuchmentPinchover,KuchmentRaich}, homogenization~\cite{SistaTewary}, impurity eigenvalues created by localized perturbations~\cite{BirmanNonregular,ZelenkoVirtual},
and Anderson localization near internal band edges~\cite{KloppWeakDisorder,VeselicLocalization}.

Part~\textup{(i)} follows from the theorem of Klopp--Ralston~\cite{KloppRalston}. We prove parts~\textup{(ii)}--\textup{(iii)} by showing that, for a fixed band $E_n(V,k)$, a generic perturbation of the potential can make the Hessian nondegenerate at its global extrema.

\vspace{1.5mm}
\noindent\textbf{Proof idea and structure of the paper.} 
We briefly describe the main idea of the proof. Without loss of generality, we only give the proof for global minima. Global maxima can be treated similarly. Fix a band $E_n(V,k)$, let $k_0$ be a global minimum at which $E_n(V,k_0)$ is a simple eigenvalue. Locally near $k_0$, we write the corresponding eigenvalue and normalized Bloch eigenfunction as $E(k)$ and $u(k)$, and set
\[
K=\ker d_k^2E(k_0),\qquad Y=K^\perp.
\]
Since $k_0$ is a minimum, the Hessian is positive definite on $Y$, so the only obstruction to nondegeneracy lies in the kernel $K$. By the implicit function theorem, we can eliminate the $Y$ directions and reduce the problem to the $K$ directions. By the Feynman–Hellmann formula, for $q\in\calX$, we have
\begin{equation*}
    \partial_t |_{t=0} E(V+tq,k) = F_q(k), \qquad F_q(k) = \int_{\mathbb{R}^d/\Gamma} q(x) |u(k,x)|^2 \, dx.
\end{equation*}
Differentiating in $k$, we have
\begin{equation*}
    d F_q(k_0) \cdot \eta   = \int_{\mathbb{R}^d/\Gamma} q(x) \rho_{\eta}(x) \, dx,\qquad \rho_{\eta}=(\eta \cdot \partial_k) |u(k_0,\cdot)|^2.
\end{equation*}
The first main ingredient of the proof is to consider the first order or the second order perturbations, depending on whether $\rho_{\eta}\equiv 0$.
\begin{enumerate}
    \item In Section~\ref{sec:first-order}, we study the case that \begin{equation}\label{eq:outline-as1}\tag{A1}
        \rho_\eta \not\equiv 0\quad \text{for some } \eta \in K.
    \end{equation} 
    In this case, we may choose $q=\rho_{\eta}$ to make $dF_q(k_0)\cdot \eta >0$.
    Thus the potential perturbation changes, to first order, the gradient of the band function in a direction belonging to the Hessian kernel $K$. After eliminating the $Y$ variables, we obtain a one-parameter family of reduced band functions on $K$. For example, when $K$ is one-dimensional, suppose $E(k)$ is of the form $(k-k_0)^4$ in the $K$ direction, then the nonvanishing $dF_q(k_0)\cdot \eta >0$ gives a perturbation of the form
    \begin{equation*}
        (k-k_0)^4+t(k-k_0).
    \end{equation*}
    Thus the critical point will move, and near the new critical point, the Hessian is nonzero.
    In general, Sard's theorem implies that, for generic values of the perturbation parameter, the reduced Hessian has positive rank at every nearby local minimum. Hence we can do an induction on the Hessian corank and get a nondegenerate Hessian. 
    \item In Section~\ref{sec:second-order}, we study the case that \begin{equation}\label{eq:outline-as2}\tag{A2}
        \rho_\eta \equiv 0\quad \text{for every } \eta \in K.
    \end{equation} 
    By \eqref{eq:outline-as2}, we have $dF_q(k_0)|_K=0$ for every $q\in \calX$, so we study the second order perturbation. Let $g_t(z)$ denote the band function after eliminating the $Y$ variables for the perturbed potential $V+tq$. Then the second order perturbation
\[
\mathcal R(q):=\frac{d}{dt}\Big\rvert_{t=0}d_z^2g_t(0)
\]
consists of terms of the form
\begin{equation*}
   \quad \quad \rho_{\eta\eta} = (\eta\cdot\partial_k)^2|u(k_0,x)|^2= 2 |(\eta\cdot\partial_k) u(k_0,x)|^2+2\Re( \overline{u(k_0,x)}(\eta\cdot\partial_k)^2u(k_0,x)).
\end{equation*}
In Lemma~\ref{lem:weighted-phase}, we show that on the nodal set $Z=\{u(k_0,x)=0\}$, we have for any $\eta\in K\setminus \{0\}$,
\begin{equation}\label{eq:bloch-density}
    (\eta\cdot\partial_k) u(k_0,x)\not\equiv 0,\quad x\in Z,
\end{equation}
under the assumption~\eqref{eq:outline-as2}. Thus we can find $q\in \mathcal{X}$ with $\mathcal R(q)>0$ on $K$ and the Hessian is nondegenerate.
\end{enumerate}
The second main ingredient is to prove Lemma~\ref{lem:weighted-phase} for \eqref{eq:bloch-density}, which is proved in Section~\ref{sec:local} using the following current identity 
\begin{equation*}
    j=|u|^2 (\eta+\nabla \alpha).
\end{equation*}
Here, the current $j$ is defined by
\begin{equation*}
    j=(\eta\cdot\partial_k) J(k_0,\cdot),\quad J(k,x)=\Im\bigl(\overline{u(k,x)}\nabla_xu(k,x)\bigr)
   +k|u(k,x)|^2,
\end{equation*}
and $(\eta\cdot\partial_k) u(k_0,x)= i\alpha(x) u(k_0,x)$ for some real valued $\alpha$ defined outside the nodal set $Z$ due to the assumption \eqref{eq:outline-as2}. Assume that $(\eta\cdot\partial_k) u(k_0,x)=0$ on $Z$, then the identities
\begin{equation*}
    \nabla \cdot j=0,\quad \int_{\R^d/\Gamma} j \,dx = 0
\end{equation*}
imply that
\begin{equation*}
    \int j \cdot (\eta + \nabla \alpha) \,dx =0,
\end{equation*}
which further implies that $\nabla \alpha + \eta = 0$. Unique continuation then implies $\eta = 0$.

Section~\ref{sec:generic-band} applies the local arguments in Section~\ref{sec:first-order} and~\ref{sec:second-order} to show that nondegeneracy is open and dense whenever the band extrema are simple. Section~\ref{sec:gap-genericity} combines this with Klopp--Ralston~\cite{KloppRalston} to prove the main Theorem~\ref{thm:main}.

\vspace{1.5mm}
\noindent\textbf{Related work.} 
The Spectral Edge Conjecture is a long-standing problem in mathematical physics. An early formulation appears in the work of Colin de Verdi\`ere
~\cite{Colin1991}, who established the Morse property for finitely many bands of small generic potentials in dimension two. The conjecture was later
formulated for periodic elliptic operators by
Kuchment--Pinchover~\cite[Conjecture~20]{KuchmentPinchover}. See also
~\cite[Conjecture~5.25]{KuchmentSurvey} and the recent
survey~\cite[Conjecture~8.5]{Kuchment2023}. Klopp--Ralston~\cite{KloppRalston} proved generic simplicity of the endpoints at a spectral gap. Filonov--Kachkovskiy~\cite{FilonovKachkovskiy} proved that global band extrema are isolated for a wide class of two-dimensional periodic elliptic operators, while Kirsch--Simon~\cite{KirschSimon} proved uniqueness and a quantitative nondegeneracy at the bottom of the spectrum. Further perturbative results include~\cite{Colin1991,ParnovskiShterenberg2017,SistaTewary}.

There is a parallel theory for discrete models. Do--Kuchment--Sottile~\cite{DoKuchmentSottile} established a genericity dichotomy for nondegenerate band extrema and verified the Spectral Edge Conjecture for a specific $\mathbb Z^2$-periodic graph. Using irreducibility of the Fermi variety, Liu~\cite{LiuFermi} proved that for discrete periodic Schr\"odinger operators an extremal level set is finite in dimension two and has dimension at most $d-2$ in dimensions $d\geq3$. See also \cite{FilonovKachkovskiy2024} for a different proof for non-divisible lattices. Faust--Liu--Luo~\cite{FaustLiuLuo} subsequently obtained an improved quantitative cardinality bound in dimension two. Berkolaiko--Canzani--Cox--Marzuola~\cite{BerkolaikoCanzaniCoxMarzuola} gave local-to-global criteria for extrema of dispersion bands on periodic graphs, and Faust--Sottile~\cite{FaustSottile} verified the Spectral Edge Conjecture for two infinite families of periodic graphs.

The nondegenerate edge hypothesis also has important consequences: Kuchment--Pinchover~\cite{KuchmentPinchover} related it to Liouville theorems, and Kuchment--Raich~\cite{KuchmentRaich} derived Green function asymptotics at internal spectral edges under the hypothesis. In dimension two, Parnovski--Shterenberg \cite{ParnovskiShterenberg2017} proved the nondegeneracy of band edges by enlarging the period lattice, and state in~\cite[Remark~1.2]{ParnovskiShterenberg2017} that nondegeneracy under the same lattice perturbation was not known. Stable degenerate edges occur for periodic quantum graphs~\cite{BerkolaikoKha} and for discrete periodic operators~\cite[Section~7]{FilonovKachkovskiy}; these examples reflect structural constraints of the discrete models and do not contradict the continuum result proved here.

\vspace{1.5mm}
\noindent\textbf{AI disclosure.} 
In June 2026, the first version of the manuscript was generated using ChatGPT 5.6, showing that band edges for generic potentials are of codimension two. The authors subsequently realized that the argument can be modified to show the generic nondegeneracy of the Hessian at band edges by generalizing an early version of Lemma~\ref{lem:weighted-phase} in the first version of the manuscript and treating the perturbation separately into two cases, as in Sections~\ref{sec:first-order}--\ref{sec:second-order}. In this process, ChatGPT 5.6 was used to compute perturbation formulas, to identify the obstruction to the nondegeneracy, to provide the proof of the key Lemma~\ref{lem:weighted-phase} and to assist substantially with the perturbation arguments in Sections~\ref{sec:first-order}--\ref{sec:gap-genericity}. In particular, the authors proposed to use the openness of the Hessian nondegeneracy to reduce the problem to showing the density, which led to a substantial simplification of the proof. A preliminary draft of this paper was generated and simplified through many interactions with ChatGPT 5.6, then verified and rewritten by the authors. ChatGPT 6 was used in the subsequent revisions to proofread the manuscript.

\vspace{1.5mm}
\noindent\textbf{Acknowledgments.}
The authors thank Ilya Kachkovskiy and Leonid Parnovski for helpful discussions. M.Y. acknowledges support from NSF grant DMS-2554813.

\section{Analytic perturbation theory and a current identity}
\label{sec:local}

In this section, we prove a Lemma~\ref{lem:weighted-phase} using a current identity, which is the key to the perturbation theory.

Let $V_0 \in \calX$, $k_0 \in \T_*^d$, and let $E_0$ be a simple eigenvalue of $H_{V_0}(k_0)$. Choose a positively oriented circle $\gamma\subset\C$ which encloses $E_0$ and no other point of $\spec H_{V_0}(k_0)$. In a neighborhood $\calV\times U$ of $(V_0,k_0)\in \calX \times \T_*^d$, the circle $\gamma$ lies in the resolvent set of $H_W(k)$ for all $(W,k)\in\calV\times U$, and the spectral projection
\begin{equation}
\label{eq:Riesz-projection}
  \Pi(W,k)=\frac{1}{2\pi i}\oint_\gamma(z-H_W(k))^{-1} \,dz
\end{equation}
has rank one. Define
\begin{equation}
\label{eq:local-energy}
  E(W,k)
  =\tr\left(\frac{1}{2\pi i}\oint_\gamma z(z-H_W(k))^{-1} \,dz\right).
\end{equation}
Then $E(W,k)$ is the unique eigenvalue of $H_W(k)$ inside $\gamma$. It is continuous in $(W,k)$, analytic in $k$, and jointly analytic when $W$ is restricted to a finite-dimensional affine subspace of $\calX$. If $U$ and the finite-dimensional parameter set are contractible, one can choose a normalized eigenfunction $u(W,k)$ which is analytic in the finite-dimensional parameters and satisfies
\begin{equation}
\label{eq:eigenvalue-equation}
  H_W(k)u(W,k)=E(W,k)u(W,k),
 \quad
  \|u(W,k)\|_{L^2(\OmegaCell)}=1.
\end{equation}

The following lemma is a consequence of standard elliptic estimates. 

\begin{lemma}
\label{lem:bounded-regularity}
Let $U\subset\R^d$ be open. Assume that $E:U\to\R$ and $u:U\to H^2_{}(\OmegaCell)$ are analytic and satisfy
\[
  H_V(k)u(k)=E(k)u(k),
 \quad
  \|u(k)\|_{L^2(\OmegaCell)}=1.
\]
For every multi-index $\beta$, every $1<p<\infty$ and every $0<{\sigma}<1$, 
\[
  \partial_k^\beta u(k,\cdot)\in W^{2,p}_{}(\OmegaCell) \quad \text{and} \quad  \partial_k^\beta u(k,\cdot)\in C^{1,{\sigma}}(\OmegaCell) 
\]
locally uniformly in $k\in U$. 
\end{lemma}

Let $E(k)$ and $u(k)$ satisfy the assumptions of Lemma~\ref{lem:bounded-regularity}, and let $k_0$ be a critical point of $E$. For $\xi,\zeta\in\R^d$, write
\begin{equation}
\label{eq:density-derivatives}
  \rho_\xi \coloneqq (\xi\cdot\partial_k)|u(k_0,\cdot)|^2,
 \quad
  \rho_{\xi\zeta} \coloneqq (\xi\cdot\partial_k)(\zeta\cdot\partial_k)|u(k_0,\cdot)|^2.
\end{equation}
These functions are real valued, continuous, bounded, and $\Gamma$-periodic.

\begin{lemma}
\label{lem:weighted-phase}
Let $\eta\in\ker d_k^2E(k_0)$, and denote
\[
  u=u(k_0,\cdot),
  \quad
  v \coloneqq (\eta\cdot\partial_k)u(k_0,\cdot),
  \quad
  Z \coloneqq \{x\in\OmegaCell:u(x)=0\}.
\]
If
\begin{equation}
\label{eq:phase-hypotheses}
  (\eta\cdot\partial_k)|u(k_0,\cdot)|^2=0
  \quad \text{on } \OmegaCell,
 \quad
  v=0
  \quad \text{on } Z,
\end{equation}
then we have $\eta=0$.
\end{lemma}

\begin{proof}
Lemma~\ref{lem:bounded-regularity} gives $u,v\in W^{2,p}\cap C^{1,{\sigma}}(\OmegaCell)$ for every $0<\sigma<1$ and $1<p<\infty$. The function $e^{ik_0\cdot x}u(x)$ solves
\[
  (-\Delta+V-E(k_0))\bigl(e^{ik_0\cdot x}u(x)\bigr)=0
\]
on $\R^d$. A nontrivial solution of this equation cannot vanish on a set of positive Lebesgue measure: at a density point it vanishes to infinite order by the argument of~\cite{deFigueiredoGossez}, and strong unique continuation~\cite{Aron,Carle} then implies that it vanishes identically. Hence
\begin{equation}
\label{eq:nodal-measure-zero}
  |Z|=0.
\end{equation}

Set $D=(\OmegaCell)\setminus Z$. The first condition in~\eqref{eq:phase-hypotheses} is
$2\Re(\overline u v)=0$. Therefore, we have
\begin{equation}
\label{eq:alpha-definition}
  v=i\alpha u \quad \text{on } D,
 \quad \quad
  \alpha=-iv/u \in C^1(D;\R).
\end{equation}
We define
\begin{equation}
\label{eq:current}
\begin{gathered}
J(k,x) \coloneqq \Im \bigl(\overline{u(k,x)}\nabla_x u(k,x)\bigr) +k|u(k,x)|^2, \\
j \coloneqq ( \eta \cdot \partial_k)J(k_0,\cdot) = \Im \bigl( \overline{v}\nabla_x u + \overline{u}\nabla_x  v\bigr) + \eta |u|^2+2k_0 \Re(\overline{u} v).   
\end{gathered}
\end{equation}
Multiplying~\eqref{eq:eigenvalue-equation} by $\overline{u(k,x)}$ and taking imaginary parts gives
\[
  \nabla_x\cdot J(k,\cdot)=0.
\]
{The Feynman--Hellmann identity in $k$-variable is given by}
\[
  \nabla_kE(k)=2\int_{\OmegaCell}J(k,x)\,dx.
\]
{Differentiating this identity and $\nabla_x\cdot J(k,\cdot)=0$ in the direction $\eta$ at $k_0$ gives}
\begin{equation}
\label{eq:current-identities}
  \nabla_x\cdot j=0,
 \quad
  \int_{\OmegaCell}j\,dx
  =\frac12d_k^2E(k_0)\eta=0.
\end{equation}

Differentiating~\eqref{eq:current} and using~\eqref{eq:alpha-definition} gives
\begin{equation}
\label{eq:j-phase}
  j=|u|^2(\eta+\nabla\alpha) \quad \text{on } D.
\end{equation}
The second condition in~\eqref{eq:phase-hypotheses} implies $j=0$ on $Z$. Since $u$ and $v$ are Lipschitz and vanish on $Z$, while their gradients are bounded,
\begin{equation}
\label{eq:j-distance}
  |j(x)|\leq C\dist(x,Z),
 \quad x\in\OmegaCell.
\end{equation}
We claim that
\begin{equation*}
    \int_D j \cdot (\eta+\nabla \alpha) \, dx =0,
\end{equation*}
which formally follows from~\eqref{eq:current-identities}. But we need justify integration by parts. If $Z=\varnothing$, take $\chi_\varepsilon=1$. Otherwise choose periodic Lipschitz functions $\chi_\varepsilon:\OmegaCell\to [0,1]$ such that
\begin{equation}
\label{eq:nodal-cutoff}
  \chi_\varepsilon=0\ \text{on }\{\dist(x,Z)\leq\varepsilon\},
 \quad
  \chi_\varepsilon=1\ \text{on }\{\dist(x,Z)\geq2\varepsilon\},
 \quad
  |\nabla\chi_\varepsilon|\leq C\varepsilon^{-1}.
\end{equation}
We choose the cutoffs so that $\chi_\varepsilon(x)\to 1$ monotonically for every $x\in D$ as $\varepsilon\to 0$.
Choose bounded smooth functions $T_N:\R\to\R$ satisfying
\begin{equation}
\label{eq:truncations}
  T_N(0)=0,
 \quad
  0\leq T_N'\leq1,
 \quad
  T_N'(s)  \nearrow 1 , \quad  s\in\R.
\end{equation}
The truncation is needed because $\alpha=-iv/u$ need not be bounded near $Z$. The first identity in~\eqref{eq:current-identities} gives
\begin{equation}
\label{eq:truncated-integration}
  \int_D\chi_\varepsilon T_N'(\alpha)j\cdot\nabla\alpha\,dx
  =-\int_D T_N(\alpha)j\cdot\nabla\chi_\varepsilon\,dx.
\end{equation}
For fixed $N$, the right hand side is bounded by
\[
  C_N\bigl|\{x:\dist(x,Z)<2\varepsilon\}\bigr|,
\]
which tends to zero by~\eqref{eq:nodal-measure-zero}. From~\eqref{eq:j-phase}, we have
\[
  j\cdot\nabla\alpha
  =|u|^2|\eta+\nabla\alpha|^2-\eta\cdot j.
\]
Let $\varepsilon\to0$ and then $N\to\infty$ in~\eqref{eq:truncated-integration}, using the monotone convergence theorem for the term on the left, and the dominated convergence theorem for the term on the right, we have
\begin{equation}
\label{eq:phase-energy}
  \int_D|u|^2|\eta+\nabla\alpha|^2\,dx
  =\eta\cdot\int_{\OmegaCell}j\,dx=0.
\end{equation}
Thus
\begin{equation}
\label{eq:alpha-gradient}
  \nabla\alpha=-\eta\quad\text{on }D.
\end{equation}

We now lift all periodic functions to $\R^d$. On a nonempty connected component $D_*$ of the lifted set $\{u\neq0\}$,
\begin{equation}
\label{eq:alpha-affine}
  \alpha(x)=C_0-\eta\cdot x.
\end{equation}
Take $D_{k_0}=-i\nabla+k_0$ and $L=D_{k_0}^2+V-E(k_0)$. Differentiating the eigenvalue equation in the direction $\eta$ and using $dE(k_0) \cdot \eta =0$ gives
\begin{equation}
\label{eq:v-equation}
  Lv=-2\eta\cdot D_{k_0} u.
\end{equation}
For $g(x)=i(C_0-\eta\cdot x)$, the commutator identity $[D_{k_0,j},g]=-\eta_j$ gives
\begin{equation}
\label{eq:gu-equation}
  L(gu)=-2\eta\cdot D_{k_0} u.
\end{equation}
The function $v-gu$ solves the homogeneous equation and vanishes on the open set $D_*$. Unique continuation~\cite{Aron} gives $v=gu$ on $\R^d$. For every $\gamma\in\Gamma$, periodicity of $u$ and $v$ gives
\[
  (\eta\cdot\gamma)u(x)=0,
 \quad  x\in\R^d.
\]
Since $u\not\equiv0$ and $\Gamma$ spans $\R^d$, we obtain $\eta=0$.
\end{proof}
\begin{remark}
    The formula~\eqref{eq:phase-energy} is related to the formula of the Hessian of the lowest Bloch band at $k=0$ in \cite[Proposition 3.7]{CV}.
\end{remark}

\section{First order perturbation}
\label{sec:first-order}
In this section, we discuss the first order perturbation, which works when $\rho_{\eta}\not\equiv 0$. For $V\in\calX$, assume that
\begin{equation}
\label{eq:minimum-set}
  k_0\in \calM_n(V) \coloneqq \argmin_{k\in\T_*^d}E_n(V,k),
\end{equation}
and $E_n(V,k_0)$ is a simple eigenvalue of $H_V(k_0)$. Let $E(W,k)$ be the eigenvalue defined by~\eqref{eq:local-energy} on $\calV\times U$. We use the notation
\[
  E(k)=E(V,k),
 \quad
  u(k)=u(V,k).
\]
Since the neighbourhood $U$ is sufficiently small, $E(k)=E_n(V,k)$ for $k\in U$. Define
\begin{equation}
\label{eq:HY-positive}
  K \coloneqq \ker d_k^2E(k_0),\quad Y\coloneqq K^\perp,\quad H_Y\coloneqq d_k^2E(k_0)|_{Y\times Y}>0,
\end{equation}
as $k_0$ is a minimum. By the Feynman--Hellmann formula, for $q\in \calX$,
\begin{equation}
\label{eq:FH-potential}
  \left.\partial_t\right|_{t=0}E(V+tq,k)=F_q(k),\quad   F_q(k) \coloneqq \int_{\OmegaCell}q(x)|u(k,x)|^2\,dx.
\end{equation}
For $\eta\in K$, recall the definition of $\rho_\eta$ in~\eqref{eq:density-derivatives}. Then
\begin{equation}
\label{eq:dF-rho}
  dF_q(k_0) \cdot \eta 
  =\int_{\OmegaCell}q\rho_\eta\,dx.
\end{equation}
In particular, if $\rho_\eta\not\equiv0$, we may take $q=\rho_\eta \in \calX$ such that 
\[
  dF_q(k_0) \cdot \eta =\|\rho_\eta\|_{L^2(\OmegaCell)}^2>0.
\]

\begin{proposition}
\label{prop:first-order-corank}
If $r = \dim K >0$ and $q\in\calX$ satisfy
\begin{equation}
\label{eq:first-order-assumption}\tag{A1'}
  dF_q(k_0)|_K\neq0,
\end{equation}
then there exist a neighborhood $U_0\Subset U$ of $k_0$, a neighborhood $\calV_0\subset\calV$ of $V$, and $\varepsilon_0>0$ such that for every $W\in\calV_0$, there exists an open dense set $\calG_{W,\varepsilon_0}\subset(-\varepsilon_0,\varepsilon_0)$ of full Lebesgue measure such that for any $t\in\calG_{W,\varepsilon_0}$, we have
\begin{equation}
\label{eq:first-order-corank-bound}
  \dim\ker d_k^2E(W+tq,k) \leq r-1
\end{equation}
at critical points $k\in U_0$ satisfying
\begin{equation}
\label{eq:critical-psd}
  d_kE(W+t q,k)=0,
 \quad
  d_k^2E(W+ t q,k) \geq 0.
\end{equation}
\end{proposition}

\begin{proof}
Use a local coordinate chart at $k_0$ and write
\[
  k=k_0+y+z,
 \quad
  y\in Y,
  \quad
  z\in K.
\]
Taking some neighborhood $\calV_0\subset \calV$ of $V$ and $\varepsilon_0>0$, for any fixed $W\in\calV_0$ and $|t|<\varepsilon_0$, we define
\[
  f(W,t,y,z) = E(W+tq,k_0+y+z).
\]
By~\eqref{eq:HY-positive}, we may choose closed small balls $\overline B_Y\subset Y$ and $\overline B_K\subset K$ centered at zero and shrinking $\calV_0$, and $\varepsilon_0$ if needed, such that $\partial_y^2f$ is positive definite on
$\calV_0\times(-\varepsilon_0,\varepsilon_0)\times\overline B_Y\times\overline B_K$. By the implicit function theorem, there exists a function $  \psi(W,t,z)\in B_Y$ such that
\begin{equation}
\label{eq:psi-first}
  \partial_yf(W,t,\psi(W,t,z),z)=0,\quad z\in\overline B_K.
\end{equation}
The positivity of $\partial_y^2f$ implies the uniqueness of $\psi(W,t,z)$ as a solution of \eqref{eq:psi-first} in $B_Y$. Hence, for $k=k_0+y+z$ with $(y,z)\in B_Y\times B_K$ such that $d_kE(W+tq,k)=0$, we have $y=\psi(W,t,z)$. 

By \eqref{eq:first-order-assumption}, we may choose $e\in K$ such that $dF_q(k_0) \cdot  e \neq0$. For simplicity, we define
\begin{equation}
\label{eq:g-Phi-first}
  g(W,t,z)\coloneqq f(W,t,\psi(W,t,z),z),\quad \Phi(W,t,z) \coloneqq d_zg(W,t,z) \cdot e .
\end{equation}
Differentiating~\eqref{eq:psi-first} in $z$ at $(V,0,0)$ and taking $\zeta\in K$ gives
\[
  H_Y(d_z\psi(V,0,0) \zeta,\cdot)
  +d_k^2E(k_0)
  (\zeta,\cdot)\big|_Y=0.
\]
By \eqref{eq:HY-positive} and that $K$ is the kernel of $d_k^2E(k_0)$, we have
\begin{equation}
\label{eq:Dzpsi-zero}
  d_z\psi(V,0,0)=0.
\end{equation}
Differentiating the identity
\[\partial_t g(V,t,z) \rvert_{t=0} = F_q (k_0+\psi(V,0,z)+z)\]
in $z$ and using~\eqref{eq:FH-potential},~\eqref{eq:Dzpsi-zero} and taking $e\in K$, we obtain
\begin{equation}
\label{eq:partial-a-Phi}
  \partial_t\Phi(V,0,0)= dF_q(k_0)\cdot (d_z \psi(V,0,0) e + e) = dF_q(k_0)\cdot e \neq0.
\end{equation}
After shrinking the neighborhood $\calV_0 \times (-\varepsilon_0, \varepsilon_0) \times B_Y \times B_K$ of $(V,0,0)$, $\partial_t\Phi$ does not vanish on it. Hence, for $W\in\calV_0$, the map $(t,z)\mapsto\Phi(W,t,z)$ is a submersion, and the set
\[
  \Sigma_W=\{(t,z):\Phi(W,t,z)=0\}
\]
is a smooth submanifold of $\R\times K$. 

Let $\pi:\Sigma_W\to\R$ be the projection $\pi(t,z)=t$. A point $(t,z)\in\Sigma_W$ is critical for $\pi$ if and only if
\[
  d_z\Phi(W,t,z)=0.
\]
Indeed, the tangent space of $\Sigma_W$ is the kernel of $\partial_t\Phi\,dt+d_z\Phi$, and $\partial_t\Phi\neq0$. Using Sard's theorem on the projection $\pi$, for almost every $t\in (-\varepsilon_0,\varepsilon_0)$, we have
\begin{equation}
\label{eq:regular-slice-first}
  d_z \Phi(W,t,z) \neq 0,
  \quad  z \in \overline B_K,\quad \Phi(W,t,z) = 0.
\end{equation}
Since this is an open condition, the corresponding set $\calG_{W,\varepsilon_0}\subset(-\varepsilon_0,\varepsilon_0)$ is open, dense, and has full Lebesgue measure.

For $t\in\calG_{W,\varepsilon_0}$ and $k=k_0+\psi(W,t,z)+z$ satisfying~\eqref{eq:critical-psd}, we have $d_zg(W,t,z)=0$. Hence $\Phi(W,t,z)=0$. Let
\[
  H=d_z^2g(W,t,z)\geq0.
\]
Since $\Phi= d_zg\cdot e$,
\begin{equation}
\label{eq:DzPhi-H}
  d_z\Phi(W,t,z) v = H(v,e),
 \quad v\in K.
\end{equation}
By~\eqref{eq:regular-slice-first}, $d_z\Phi(W,t,z)\neq0$. Hence, we have
$H(\,\cdot\,,e)\neq0$ and $H$ has rank at least one. Since
$\dim K=r$, for the Hessian $H$ of the reduced band function $g$, we have
\[
  \dim\ker H\leq r-1.
\]

To compare $H$ with the full Hessian, write the full Hessian of $E(k)$ on $Y\oplus K$ as
\[
  \begin{pmatrix}A&B\\B^{\mathsf T}&C\end{pmatrix},
 \quad A>0.
\]
The Hessian of $g$ is the Schur complement
\[
  H=C-B^{\mathsf T}A^{-1}B.
\]
The block-matrix factorization
\[
  \begin{pmatrix} A &B\\
  B^{\mathsf T} &C
  \end{pmatrix}
  =
  \begin{pmatrix}
  I &0\\
  B^{\mathsf T}A^{-1} &I
  \end{pmatrix}
  \begin{pmatrix}
  A &0\\
  0 &H
  \end{pmatrix}
  \begin{pmatrix}
  I &A^{-1}B\\
  0 &I
  \end{pmatrix}
\]
shows that the kernel of the full Hessian has the same dimension as that of $H$. This proves~\eqref{eq:first-order-corank-bound}.
\end{proof}

\section{Second order perturbation}
\label{sec:second-order}

In this section, we follow the notation of Section~\ref{sec:first-order} and focus on the case 
\begin{equation}
\label{eq:rho-vanishes}\tag{A2}
  \rho_\eta=0
 \quad\text{for every }\eta\in K.
\end{equation}
For $q\in\calX$, define
\[
  f_t(y,z)\coloneqq E(V+tq,k_0+y+z).
\]
As in \eqref{eq:psi-first}, let $\psi(t,z)$ be the function such that
\begin{equation}
\label{eq:psi-second}
  \partial_yf_t(\psi(t,z),z)=0,
\end{equation}
and, for simplicity, we write
\begin{equation}
\label{eq:g-second}
  g_t(z) \coloneqq f_t(\psi(t,z),z).
\end{equation}
We also introduce the notations
\begin{equation}
\label{eq:kappa-R}
  \kappa(q)\coloneqq -H_Y^{-1}\bigl(dF_q(k_0)|_Y\bigr),
 \quad
  \mathcal R(q)
  \coloneqq d^2F_q(k_0)|_{K\times K}
   +d^3E(k_0)(\kappa(q),\cdot,\cdot)|_{K\times K}.
\end{equation}
The following lemma gives the infinitesimal change of the Hessian in the $K$ direction.
\begin{lemma}
For every $q\in\calX$,
\begin{equation}
\label{eq:R(q)}
  \left. \partial_t \right|_{t=0} d_z^2 g_t(0)= \mathcal{R}(q),
\end{equation}
where for $\eta,\zeta\in K$, 
\begin{equation}
\label{eq:R-Theta}
  \mathcal{R}(q) (\eta,\zeta)
  =\int_{\OmegaCell} q(x) \Theta_{\eta,\zeta} (x) \, dx,
\end{equation}
with
\begin{equation}
\label{eq:Theta}
  \Theta_{\eta,\zeta}
  =\rho_{\eta\zeta} + \rho_{\chi(\eta,\zeta)}, 
  \quad
  \chi(\eta,\zeta)=-H_Y^{-1}C(\eta,\zeta),
  \quad 
  C(\eta,\zeta)[y]=d^3E(k_0)(y,\eta,\zeta).
\end{equation}
\end{lemma}

\begin{proof}
Differentiating~\eqref{eq:psi-second} in $t$ at $(t,z)=(0,0)$, we get
\begin{equation}
\label{eq:dtpsi-kappa}
  \partial_t\psi(0,0)=\kappa(q).
\end{equation}
 As in~\eqref{eq:Dzpsi-zero}, 
 \[
   d_z\psi(0,0)=0.
 \]
Differentiating~\eqref{eq:psi-second} twice in $z$ and pairing with $\eta,\zeta\in K$ gives
\begin{equation}
\label{eq:Dzzpsi-chi}
  d_z^2\psi(0,0)(\eta,\zeta)=\chi(\eta,\zeta).
\end{equation}
Differentiating with respect to $t$ gives
\begin{align}
\label{eq:partial-gt}
  \partial_t g_t(z) 
  & = \partial_t E_t(k_0+\psi(t,z)+z) +dE_t(k_0+\psi(t,z)+z) \cdot \partial_t\psi(t,z) \\
  & = \partial_t E_t(k_0+\psi(t,z)+z). \nonumber
\end{align}
where the second term in \eqref{eq:partial-gt} vanishes by \eqref{eq:psi-second} and 
$\partial_t\psi(t,z)\in Y$. Evaluating at $t=0$, 
\begin{equation}\label{eq:dt0gt}
    \left.\partial_t\right|_{t=0}g_t(z) = F_q(k_0+\psi(0,z)+z).
\end{equation}
Taking two derivatives in $z$ at zero and using~\eqref{eq:Dzzpsi-chi} gives
\[
  \left.\partial_t\right|_{t=0}d_z^2g_t(0)(\eta,\zeta)
  =d^2F_q(k_0)(\eta,\zeta)
   +dF_q(k_0)\cdot \chi(\eta,\zeta).
\]
By symmetry of $H_Y^{-1}$, the last term is
$d^3E(k_0)(\kappa(q),\eta,\zeta)$. This proves~\eqref{eq:R(q)}. 
\end{proof}

The following lemma gives nondegenerate Hessian under perturbation.

\begin{lemma}
\label{lem:second-order-direction}
Assume that $K\neq\{0\}$ and~\eqref{eq:rho-vanishes} holds, then there exists $q\in\calX$ such that
\begin{equation}
\label{eq:R-positive}
  \mathcal R(q)>0
 \quad\text{on }K.
\end{equation}
\end{lemma}

\begin{proof}
Let $Z=\{u(k_0,\cdot)=0\}$.
By Lemma~\ref{lem:weighted-phase}, for any $\eta\in K$ with $|\eta|=1$, there exists $x_{\eta}\in Z$ such that
\begin{equation}
    \eta\cdot \partial_k u(k_0,x_{\eta})\neq 0.
\end{equation}
Note that for $x\in Z$, we have
\begin{equation*}
    \rho_{\eta \eta }(x)
  =2|\eta \cdot\partial_k u(k_0,x)|^2,
  \quad
  \rho_{\chi(\eta,\eta)}(x)=0.
\end{equation*}
Consequently,
\begin{equation}
\label{eq:Theta-on-Z}
  \Theta_{\eta,\eta}(x)
  =2|\eta\cdot\partial_ku(k_0,x)|^2,
  \quad x\in Z.
\end{equation}
By continuity of $\partial_k u(k_0, x)$ and compactness of $\{|\eta|=1\}$, there exist $\delta>0$, $\eta_j\in K$ such that
\begin{equation}\label{eq:covering}
    |\eta_j|=1,\quad \{|\eta|=1\}\subset \cup_{j} B(\eta_j,\delta),
\end{equation}
and $q_j\in \mathcal{X}$ with $q_j\geq 0$ supported near $x_{\eta_j}\in Z$ for $j=1,2,\ldots,N$ ($q_j$ can be taken to be a multiple of approximation of identities) such that
\begin{equation*}
    \mathcal{R}(q_{j})(\eta,\eta) \geq \left\{\begin{array}{ll}
    1, & |\eta-\eta_j|<\delta,  \\
    -\epsilon,    & |\eta|=1,
    \end{array}\right.
\end{equation*}
for $\epsilon\ll N^{-1}$. Let $q=\sum_j q_j$, then
for any $\eta\in K$ with $|\eta|=1$, choose $j_0$ such that $|\eta-\eta_{j_0}|<\delta$ by~\eqref{eq:covering}, we have
\begin{equation*}
    \mathcal{R}(q)(\eta,\eta)\geq \mathcal{R}(q_{j_0})(\eta,\eta)+\sum_{j\neq j_0}\mathcal{R}(q_{j})(\eta,\eta) \geq 1-N\epsilon>0. \qedhere
\end{equation*}
\end{proof}

With Lemma~\ref{lem:second-order-direction}, we can state the main second order perturbation proposition.
\begin{proposition}
\label{prop:second-order-minimum}
Assume that $K\neq\{0\}$, and let $q\in\calX$ satisfying $\mathcal R(q)>0$ on $K$ and 
\begin{equation}
\label{eq:second-order-assumptions} \tag{A2'}
  dF_q(k_0)|_K=0.
\end{equation}
There is a neighborhood $U_0$ of $k_0$ and $t_0>0$ such that for every $0<t<t_0$ the function $k\mapsto E(V+tq,k)$ attains its unique minimum over $\overline{U}_0$ at $k_t\in U_0$ with
\begin{equation}
\label{eq:second-order-positive-Hessian}
  d_k^2E(V+tq,k_t)>0.
\end{equation}
\end{proposition}

\begin{proof}
Let
\begin{equation}
\label{eq:Gq}
  G_q(z) \coloneqq \left.\partial_t\right|_{t=0} g_t(z).
\end{equation}
The first condition in~\eqref{eq:second-order-assumptions}, \eqref{eq:dt0gt} and~\eqref{eq:R(q)} give
\begin{equation}
\label{eq:Gq-derivatives}
  d G_q(0)=0,
 \quad
  d^2 G_q(0) = \mathcal{R}(q)>0.
\end{equation}
Choose a closed ball ${\overline B}\subset K$ centered at zero such that
\begin{equation}
\label{eq:g0-Gq-bounds}
  g_0(z) \geq g_0(0),
 \quad
  G_q(z)-G_q(0) \geq c|z|^2,
 \quad z\in\overline B,
\end{equation}
for some $c>0$. Taking the $Y$-neighborhood sufficiently small, we may assume that $\partial_y^2 f_t$ is uniformly positive definite. Then, for each $z\in\overline B$ and sufficiently small $t$, the point $y=\psi(t,z)$ is the unique minimum of $y\mapsto f_t(y,z)$ in that neighborhood. Moreover, minima of $E(V+tq,\cdot)$ over a fixed neighborhood $U_0$ correspond exactly to minima of $g_t$ over $\overline B$.

Taylor expansion in $t$ gives
\begin{equation}
\label{eq:gt-expansion}
  g_t(z)=g_0(z) + t G_q(z) + t^2 S(t,z)
\end{equation}
uniformly in $C^2(\overline{B})$. By~\eqref{eq:g0-Gq-bounds}, we have $g_t(z)>g_t(0)$ for all sufficiently small $t>0$ and $z\in \partial B$. Thus every minimizer $z_t$ lies in $B$. Since $S$ is uniformly Lipschitz in $z$,
\[
  |S(t,z)-S(t,0)| \leq C |z|.
\]
\eqref{eq:g0-Gq-bounds} and \eqref{eq:gt-expansion} give
\[
  0\geq g_t(z_t)-g_t(0)
  \geq ct|z_t|^2-C t^2 |z_t|.
\]
Therefore
\begin{equation}
\label{eq:zt-Ot}
  |z_t|\leq C_1t.
\end{equation}
Since $K=\ker d_k^2E(k_0)$, we have $d_z^2g_0(0)=0$. Since $g_0(z)\geq g_0(0)$, we must have 
\begin{equation}
\label{eq:g0-Hessian-quadratic}
  d_z^3g_0(0)=0, \quad d_z^2g_0(z)=O(|z|^2).
\end{equation}
Using~\eqref{eq:Gq-derivatives},~\eqref{eq:gt-expansion} and~\eqref{eq:g0-Hessian-quadratic}, we obtain uniformly for $|z|\leq C_1t$,
\begin{equation}
\label{eq:gt-Hessian-positive}
  d_z^2g_t(z)
  =t\mathcal R(q)+O(t^2)>0
\end{equation}
for $t$ sufficiently small. The ball $\{|z|\leq C_1t\}$ is convex and contains every minimizer by~\eqref{eq:zt-Ot}; thus~\eqref{eq:gt-Hessian-positive} implies the uniqueness of $z_t$. The Schur complement argument similar to Proposition~\ref{prop:first-order-corank} gives~\eqref{eq:second-order-positive-Hessian}.
\end{proof}

\section{Generic nondegeneracy for a single band}
\label{sec:generic-band}

In this section, we combine Section~\ref{sec:first-order} and Section~\ref{sec:second-order} to show generic nondegeneracy of the band edge for a single band. For $n\geq1$, define
\begin{equation}
\label{eq:Snmin}
  \mathcal S_n^{\min} \coloneqq \big\{V\in\calX:
      E_n(V,k)\text{ is a simple eigenvalue of }H_V(k)
      \text{ for all }k\in\calM_n(V)\big\}.
\end{equation}

\begin{lemma}
\label{lem:simple-min-open}
The set $\mathcal S_n^{\min}$ is open in $\calX$.
\end{lemma}

\begin{proof}
Let $V\in\mathcal S_n^{\min}$ and define
\[
  m_n(V)=\min_{k\in\T_*^d}E_n(V,k).
\]
For each $k\in\calM_n(V)$, choose a neighborhood $U_k$ and a circle $\gamma_k$ as in~\eqref{eq:Riesz-projection}. Compactness of $\calM_n(V)$ gives a finite subcover $U_{k_1},\ldots,U_{k_N}$. Their union $U$ contains $\calM_n(V)$, and there exists $\delta>0$ such that
\begin{equation}
\label{eq:exterior-value-gap}
  E_n(V,k)\geq m_n(V)+3\delta,
 \quad k\in\T_*^d\setminus U.
\end{equation}
The min--max principle gives
\begin{equation}
\label{eq:minmax-Linfty}
  \sup_{k\in\T_*^d}|E_n(W,k)-E_n(V,k)|
  \leq\|W-V\|_{L^\infty}.
\end{equation}
If $\|W-V\|_{L^\infty}<\delta$, then~\eqref{eq:exterior-value-gap}--\eqref{eq:minmax-Linfty} imply that every point of $\calM_n(W)$ lies in $U$. The contours $\gamma_{k_j}$ remain in the resolvent sets for $W$ close to $V$, so the eigenvalue $E_n(W,k)$ is simple at every $k\in\calM_n(W)$ by analytic perturbation theory. 
\end{proof}

For $V\in\mathcal S_n^{\min}$, define
\begin{equation}
\label{eq:Rn}
  r_n(V) \coloneqq \max_{k\in\calM_n(V)}\dim\ker d_k^2E_n(V,k).
\end{equation}
We use the following notation at a pair $(W,k)$ with $W\in\mathcal S_n^{\min}$ and $k\in\calM_n(W)$. Let $E_{W}(\kappa)$ and $u_{W}(\kappa)$ be defined as in~\eqref{eq:local-energy} and \eqref{eq:eigenvalue-equation} near $\kappa=k$. Write 
\begin{equation}
\label{eq:varying-K-rho}
  K(W,k)=\ker d_\kappa^2E_{W}(k),
 \quad
  \rho_{W,k,\eta}
  =(\eta\cdot\partial_\kappa)|u_{W}(k,\cdot)|^2,
  \quad \eta\in K(W,k).
\end{equation}
For $q\in\calX$, let $F_{W,k,q}$ be the function in~\eqref{eq:FH-potential} formed with $u_{W}$, and let $\mathcal R_{W,k}(q)$ be the form defined by~\eqref{eq:kappa-R}.

\begin{lemma}
\label{lem:local-corank}
Let $V\in\mathcal S_n^{\min}$, and let $r=r_n(V)>0$ as in \eqref{eq:Rn}, and let $k_0\in\calM_n(V)$ satisfy
\[
  \dim\ker d_k^2E_n(V,k_0)=r.
\]
There exist a compact neighborhood $Q$ of $k_0$, with
$k_0\in\operatorname{int}Q$, and an open neighborhood
$\calV_0\subset\mathcal S_n^{\min}$ of $V$ such that the set
\begin{equation}
\label{eq:GQ}
  \calG_Q \coloneqq \bigl\{W\in\calV_0:
       \dim\ker d_k^2E_n(W,k)\leq r-1
       \text{ for every }k\in\calM_n(W)\cap Q\bigr\}
\end{equation}
is open and dense in $\calV_0$.
\end{lemma}

\begin{proof}
Write $K_0=K(V,k_0)$ defined in \eqref{eq:varying-K-rho}. In a neighborhood $\calV_0\times U$ of $(V,k_0)$, we have $\calV_0\subset\mathcal S_n^{\min}$ and
\begin{enumerate}[label=\textup{(\alph*)},leftmargin=3em]
  \item for $(W,k)\in\calV_0\times U$, the eigenvalue $E_n(W,k)$ is the unique eigenvalue inside a contour $\gamma$;
  \item if $k\in\calM_n(W)\cap U$, then $\dim K(W,k)\leq r = r_n(V)$.
\end{enumerate}
The second property follows from upper semicontinuity of the Hessian corank. 

\vspace{1mm}
\noindent\textbf{Step 1:~density.}
We first prove the density of $\calG_Q$ in $\calV_0$.

\vspace{1mm}
\noindent
\emph{Case 1.} If $\rho_{V,k_0,\eta}\not\equiv0$ for some $\eta\in K_0$, we choose $q=\rho_{V,k_0,\eta}$. Then
\[
  dF_{V,q}(k_0) \cdot \eta >0.
\]
For sufficiently small neighbourhoods $U$ and $\calV_0$, Proposition~\ref{prop:first-order-corank} applies to every potential $W\in\calV_0$. Let $Q\Subset U$ contain $k_0$ in its interior. For every $W\in\calV_0$ and every neighborhood $\mathcal W$ of $W$, Proposition~\ref{prop:first-order-corank} gives $t$ arbitrarily small such that $W+tq\in\mathcal W\cap\calG_Q\cap\calV_0$. Thus $\calG_Q$ is dense. 

\vspace{1mm}
\noindent
\emph{Case 2.} If $\rho_{V,k_0,\eta}=0$ for every $\eta\in K_0$, by Lemma~\ref{lem:second-order-direction}, we may choose $q\in\calX$ such that
\begin{equation}
\label{eq:R-base-positive}
  \mathcal R_{V,k_0}(q)>0
 \quad\text{on }K_0.
\end{equation}
For sufficiently small neighbourhoods $U$ and $\calV_0$, we have
\begin{equation}
\label{eq:R-persistence}
  \left.
  \begin{gathered}
    W\in\calV_0,\quad k\in\calM_n(W)\cap U,\\
    \dim K(W,k)=r,\quad
    dF_{W,q}(k)|_{K(W,k)}=0
  \end{gathered}
  \right\}
  \Longrightarrow
  \mathcal R_{W,k}(q)>0\text{ on }K(W,k).
\end{equation}
Indeed, \eqref{eq:R-persistence} follows from \eqref{eq:R-Theta}, \eqref{eq:R-base-positive} and the continuity of the spectral projection for $d_k^2 E_n(W,k)$ (since the kernel space $K(W,k)$ is the range of the spectral projection to small eigenvalues when $\dim K(W,k)=r$, and the Hessian remains uniformly invertible on its orthogonal complement). 

Choose $Q\Subset U$ with $k_0$ in the interior of $Q$. Let $W\in\calV_0$ and
\begin{equation}
\label{eq:BW}
  B_W
  =\{k\in\calM_n(W)\cap Q:\dim K(W,k)=r\}.
\end{equation}
The set $B_W$ is compact. If $B_W=\varnothing$, then $W\in\calG_Q$. Suppose that $B_W\neq\varnothing$.

For each $k\in B_W$, if
\begin{equation}
\label{eq:first-order-at-Wk}
  dF_{W,q}(k)|_{K(W,k)} \neq 0,
\end{equation}
then Proposition~\ref{prop:first-order-corank} gives a neighborhood $U_k$ of $k$ and an open dense subset $G_k\subset(0,\varepsilon_k)$ such that, for $s\in G_k$, every point $\kappa\in U_k$ satisfying
\[
  d_\kappa E_n(W+sq,\kappa)=0,
 \quad
  d_\kappa^2E_n(W+sq,\kappa)\geq0
\]
has Hessian corank at most $r-1$. If~\eqref{eq:first-order-at-Wk} fails, then~\eqref{eq:R-persistence} and Proposition~\ref{prop:second-order-minimum} give a neighborhood $U_k$ and $\varepsilon_k>0$ such that every global minimum in $U_k$ has Hessian corank at most $r-1$ for every $s\in G_k\coloneqq (0,\varepsilon_k)$.

By compactness of $B_W$, we can choose $k_1,\ldots,k_m\in B_W$ such that
\begin{equation}
\label{eq:BW-cover}
  B_W\subset\bigcup_{j=1}^mU_{k_j}.
\end{equation}
Therefore, the set $G=\bigcap_{j=1}^mG_{k_j}$ is open and dense in an interval $(0,\varepsilon)$.

For all sufficiently small $s>0$, every point of
\[
  \{\kappa\in\calM_n(W+sq)\cap Q:\dim K(W+sq,\kappa)=r\}
\]
belongs to the union in~\eqref{eq:BW-cover}. This continuity follows from the estimate~\eqref{eq:minmax-Linfty} and the property~\textup{(b)}.

Choose $s\in G$ arbitrarily small. Then $W+sq\in\calG_Q$, which proves the density in Case 2.

\noindent\textbf{Step 2:~openness.}
It remains to show that $\calG_Q$ is open. Suppose $W_j\to W_*$ in $\calV_0$, $W_*\in\calG_Q$, and $W_j\notin\calG_Q$. Choose
\[
  k_j\in\calM_n(W_j)\cap Q,
 \quad
  \dim K(W_j,k_j)\geq r.
\]
After passing to a subsequence, $k_j\to k_*\in Q$. The estimate~\eqref{eq:minmax-Linfty} gives $k_*\in\calM_n(W_*)$, and the $C^2$ convergence of the eigenvalues and the upper semicontinuity give
\[
  \dim K(W_*,k_*)\geq r,
\]
which contradicts to $W_*\in\calG_Q$. Thus $\calG_Q$ is open.
\end{proof}

\begin{proposition}
\label{prop:fixed-band-nondegenerate}
For every $n\geq1$, the set
\begin{equation}
\label{eq:Nnmin}
  \mathcal N_n^{\min}
  =\bigl\{V\in\mathcal S_n^{\min}:\, d_k^2E_n(V,k)>0 \text{ for every }k\in\calM_n(V)\bigr\}
\end{equation}
is open and dense in $\mathcal S_n^{\min}$.
\end{proposition}
We note that the similar statement works for band maxima with the same proof.

\begin{proof}
We first show openness. Fix $V\in\mathcal N_n^{\min}$. $\calM_n(V)$ is finite and each point of $\calM_n(V)$ is an isolated minimum. Choose pairwise disjoint neighborhoods $U_1,\ldots,U_N$ of its minima such that the eigenvalues have Hessian bounded below by a positive constant on each $U_j$. The openness follows from the continuity of the Hessian and the eigenvalues $E_n(W,\cdot)$ with respect to the perturbations in $W$.  

We now prove density. Let $\mathcal U\subset\mathcal S_n^{\min}$ be nonempty and open. Choose $V_0\in\mathcal U$ and put $r_0=r_n(V_0)$, where $r_n(V_0)$ is defined in \eqref{eq:Rn}. If $r_0=0$, then $V_0\in\mathcal N_n^{\min}$. Assume $r_0>0$, and define the compact set
\begin{equation}
\label{eq:max-corank-set}
  B_{r_0}(V_0)
  =\{k\in\calM_n(V_0):\dim\ker d_k^2E_n(V_0,k)=r_0\}.
\end{equation}
Apply Lemma~\ref{lem:local-corank} to every point of $B_{r_0}(V_0)$. By the compactness of $B_{r_0}(V_0)$, we may choose finitely many compact neighborhoods $Q_1,\ldots,Q_N$ whose interiors cover $B_{r_0}(V_0)$ such that Lemma~\ref{lem:local-corank} holds with neighbourhoods $\calV_j \times Q_j$. Then we choose an open neighborhood $\calV_0\subset \cap_j \calV_j$ of $V_0$ .

For sufficiently small $\calV_0$, every $W\in\calV_0$ and every
\[
k \in \calM_n(W) \setminus \bigcup_{j=1}^N \operatorname{int}Q_j
\]
satisfy
\begin{equation}
\label{eq:exterior-corank}
  \dim \ker d_k^2E_n(W,k)\leq r_0-1.
\end{equation}
For each $j$, Lemma~\ref{lem:local-corank} gives an open dense subset $\calG_{V_0,j}$ of $\calV_0$ on which every global minimum in $Q_j$ has Hessian corank at most $r_0-1$. Then $\calG_{V_0}=\cap_{j} \calG_{V_0,j}$ is again open and dense. Together with~\eqref{eq:exterior-corank}, there exists a nonempty open set $\mathcal U_1\subset\mathcal U$ such that
\begin{equation}
\label{eq:R-drop}
  r_n(W)\leq r_0-1,
 \quad W\in\mathcal U_1.
\end{equation}
Repeating the construction at most $d$ steps gives density.
\end{proof}

\section{Proof of Theorem \ref{thm:main}}
\label{sec:gap-genericity}

We take intervals $I=[r,s]$ with rational endpoints $r,s\in\mathbb Q.$ Define $\calV_I\subset\calX$ to be the set of $V$ for which $I$ is contained in a bounded connected component of $\R\setminus\spec H_V$. For $V\in\calV_I$, denote this component by $(a_I(V),b_I(V))$. Define $\calD_I\subset\calV_I$ by the following condition: $V \in \calD_I$ if there exist unique indices $n_-(V),n_+(V)$ such that
\begin{equation}
\label{eq:unique-endpoint-bands}
  a_I(V)\in E_{n_-(V)}(V,\T_*^d),
 \quad
  b_I(V)\in E_{n_+(V)}(V,\T_*^d).
\end{equation}
Note that the condition~\eqref{eq:unique-endpoint-bands} implies that each endpoint is a simple eigenvalue of $H_V(k)$ at every attaining quasimomentum $k$. 

\begin{lemma}
\label{lem:KR}
The set $\calV_I$ is open in $\calX$, and $\calD_I$ is open and dense relative to $\calV_I$.
\end{lemma}

\begin{proof}
The min--max principle gives
\begin{equation}
\label{eq:spectral-Hausdorff}
  \dist_H(\spec H_V,\spec H_W)
  \leq\|V-W\|_{L^\infty}.
\end{equation}
Thus $\calV_I$ is open. Applying~\cite[Theorem~1.1]{KloppRalston} at the two endpoints yields the density of $\calD_I$ in $\calV_I$. The relative openness of $\calD_I$ in $\calV_I$ again follows from \eqref{eq:spectral-Hausdorff} and the existence of a uniform band gap near band extrema.  
\end{proof}

\begin{lemma}
\label{lem:gap-global}
Let $(a,b)$ be a connected component {of $\R\setminus\spec H_V$}.
\begin{enumerate}
  \item {If $E_n(V,k_0)=a$ for some $k_0$, then} $a=\max_{k\in\T_*^d}E_n(V,k).$
  \item {If $E_m(V,k_0)=b$ for some $k_0$, then} $b=\min_{k\in\T_*^d}E_m(V,k).$
\end{enumerate}
\end{lemma}
We omit the proof as it is trivial. For $V\in\calV_I$, let $\calN_I$ be the set of potentials such that (1) $ V \in \calD_I$; (2) for every $k$ attaining $a_I(V)$, $d_k^2E_{n_-(V)}(V,k)<0;$ (3) for every $k$ attaining $b_I(V)$, $d_k^2E_{n_+(V)}(V,k)>0.$

\begin{proposition}
\label{prop:gap-open-dense}
The set $\calN_I$ is open and dense relative to $\calV_I$.
\end{proposition}

\begin{proof}
Fix $V\in\calN_I$. Lemma~\ref{lem:gap-global} identifies the lower endpoint as a global maximum of $E_{n_-(V)}$ and the upper endpoint as a global minimum of $E_{n_+(V)}$. The openness of~\eqref{eq:Nnmin} together with Lemma~\ref{lem:KR} shows that $\calN_I$ is relatively open in $\calV_I$.

To prove the density of $\calN_I$, let $\mathcal U\subset\calV_I$ be nonempty and relatively open. By Lemma~\ref{lem:KR}, $\mathcal U\cap\calD_I \neq \varnothing$. We may choose a nonempty open set
\begin{equation}
    \label{eq:U0}
     \mathcal U_0\subset\mathcal U\cap\calD_I
\end{equation}
and fixed indices $n_-,n_+$ such that~\eqref{eq:unique-endpoint-bands} holds with these indices for every $V\in\mathcal U_0$. Lemma~\ref{lem:gap-global} and the definition \eqref{eq:Snmin} give
\[
  \mathcal U_0\subset\mathcal S_{n_-}^{\max}\cap\mathcal S_{n_+}^{\min},
\]
where $\mathcal S_{n_-}^{\max}$ is defined analogously for band maxima. By Proposition~\ref{prop:fixed-band-nondegenerate},
$\mathcal N_{n_+}^{\min}$ is dense in $\mathcal{S}_{n_+}^{\min}$. Thus, there exists a nonempty open set
\begin{equation}
    \label{eq:U1}
    \mathcal U_1\subset\mathcal U_0\cap\mathcal N_{n_+}^{\min}.
\end{equation}
By the density of $\mathcal N_{n_-}^{\max}$ (defined analogously for band maxima), there exists a nonempty open set
\begin{equation}
    \label{eq:U2}
    \mathcal U_2 \subset\mathcal U_1\cap\mathcal N_{n_-}^{\max}.
\end{equation}
Thus, by \eqref{eq:U0}, \eqref{eq:U1} and \eqref{eq:U2}, we have a nonempty open set $\mathcal U_2\subset\calN_I\cap \calU$. This proves the density of $\calN_I$ in $\calV_I$.
\end{proof}

\begin{proof}[Proof of Theorem \ref{thm:main}]
Define
\begin{equation}
\label{eq:OI}
  \calO_I
  =\calN_I\cup\bigl(\calX\setminus\overline{\calV_I}\bigr).
\end{equation}
The set $\calO_I$ is open by definition and dense by Proposition~\ref{prop:gap-open-dense}. Let
\begin{equation}
\label{eq:final-residual-set}
  \calG
  =\bigcap_{\substack{r,s\in\mathbb Q\\r<s}}\calO_{[r,s]}.
\end{equation}
By definition, $\calG$ is residual. Let $V\in\calG$. For any bounded spectral gap $(a,b)$ of $\spec H_V$. Choose $r,s\in\mathbb Q$ with $a<r<s<b$. Then $V\in\calV_{[r,s]}$, and~\eqref{eq:OI}--\eqref{eq:final-residual-set} imply $V\in\calN_{[r,s]}$. This proves Parts~\textup{(i)} and~\textup{(iii)}. A nondegenerate critical point is isolated. Hence, there are only finitely many of them as $\T^d_*$ is compact. This proves Part~\textup{(ii)}.   
\end{proof}

\end{document}